\documentclass[12pt,letterpaper]{amsart}
\usepackage{geometry}
\usepackage{hyperref,amssymb}
\usepackage{mathrsfs,setspace,xcolor}
\usepackage{mathpazo} 

\usepackage{amsmath,amsthm}
\newtheorem{theorem}{Theorem}
\newtheorem{definition}[theorem]{Definition}
\newtheorem{remark}[theorem]{Remark}
\newtheorem{proposition}[theorem]{Proposition}

\newtheorem{lemma}[theorem]{Lemma}
\newtheorem{corollary}[theorem]{Corollary}
\numberwithin{equation}{theorem}

\newcommand{\C}{\mathbb{C}}
\newcommand{\D}{\Omega}
\newcommand{\ep}{\varepsilon}
\newcommand{\Dc}{\overline{\Omega}}
\newcommand{\dbar}{\overline{\partial}}
\newcommand{\zb}{\overline{z}}
\newcommand{\wb}{\overline{w}}
\onehalfspace

\title{A local weighted Axler-Zheng theorem in $\mathbb{C}^n$}

\author{\v{Z}eljko \v{C}u\v{c}kovi\'c}
\address[\v{Z}eljko \v{C}u\v{c}kovi\'c]{University of Toledo, Department of 
	Mathematics \& Statistics, Toledo, OH 43606, USA}
 \email{Zeljko.Cuckovic@utoledo.edu}

\author{S\"{o}nmez \c{S}ahuto\u{g}lu}
\address[S\"{o}nmez \c{S}ahuto\u{g}lu]{University of Toledo, Department of 	
Mathematics \& Statistics, Toledo, OH 43606, USA}
\curraddr{Sabanc\i{} University, Tuzla, Istanbul, 34956, Turkey}
\email{Sonmez.Sahutoglu@utoledo.edu}

\author{Yunus E. Zeytuncu}
\address[Yunus E. Zeytuncu]{University of Michigan--Dearborn, Department of 
	Mathematics \& Statistics, Dearborn, MI 48128, USA}
\email{zeytuncu@umich.edu}


\date{\today}
\begin{document}

\begin{abstract}
The well-known Axler-Zheng theorem characterizes compactness of finite sums of finite products of Toeplitz operators on the unit disk in terms of the Berezin transform of these operators. Subsequently this theorem was generalized to other domains and appeared in different forms, including domains in $\mathbb{C}^n$ on which the $\dbar$-Neumann operator $N$ is compact. In this work we remove the assumption on $N$, and we study weighted Bergman spaces on smooth bounded pseudoconvex domains. We prove a local version of the Axler-Zheng theorem characterizing compactness of Toeplitz operators in the algebra generated by symbols continuous up to the boundary in terms of the behavior of the Berezin transform at strongly pseudoconvex points. We employ a Forelli-Rudin type inflation method to handle the weights.

\end{abstract}
\maketitle

\section{Introduction}
\subsection{History}
In the theory of Bergman space operators on the open unit disk $\mathbb{D}$, 
Axler-Zheng theorem \cite{AxlerZheng98} provides an important characterization 
of compactness of a large class of operators in terms of their Berezin 
transforms.  Specifically this theorem states that if $S$ is a finite sum of finite 
products of Toeplitz operators on the Bergman space $A^2(\mathbb{D})$ whose 
symbols are in $L^{\infty}(\mathbb{D})$, then $S$ is compact if and only if the 
Berezin transform of $S$, $BS(z) \rightarrow 0$ as $|z|\rightarrow 1$.  This 
theorem has been extended by Suarez \cite{Suarez07} to include all operators in 
the Toeplitz algebra in the unit ball in $\mathbb{C}^n$. Englis \cite{Englis99} 
extended the Axler-Zheng theorem to irreducible bounded symmetric domains and the 
unit polydisk.  Mitkovski, Suarez and Wick \cite{MitkovskiSuarezWick13} proved 
a weighted version of Suarez's result on the unit ball in 
$\mathbb{C}^n.$ Using the techniques of several complex variables, 
\u{C}u\u{c}kovi\'c and \c{S}ahuto\u{g}lu \cite{CuckovicSahutoglu13} 
proved a version of the Axler-Zheng theorem on smooth bounded pseudoconvex 
domains on which the $\overline\partial$-Neumann operator is compact.  The use 
of the $\overline\partial$ techniques required that the operators in their 
theorem belong to the algebra $\mathscr{T}(\overline\Omega)$ which is the norm 
closed algebra generated by $\{T_{\phi}:\phi\in C(\Dc)\}.$ Recently, in her 
Master's thesis \cite{KreutzerThesis}, Kreutzer generalized  
\u{C}u\u{c}kovi\'c and  \c{S}ahuto\u{g}lu's result in a more abstract setting.

In this paper our aim is to extend the previous result of \u{C}u\u{c}kovi\'c 
and \c{S}ahuto\u{g}lu in two ways: Firstly, we want to remove the hypothesis of 
the compactness of the $\overline\partial$-Neumann operator on $\Omega$.  We 
also want to consider weighted Bergman spaces.  Our main theorem gives a local 
version of the Axler-Zheng theorem  for a wide class of domains in 
$\mathbb{C}^n.$  The novelty of our approach is to use the inflation of the 
domain argument pioneered by Forelli-Rudin and Ligocka 
\cite{ForelliRudin74/75,Ligocka89}.  The second important ingredient is the 
B-regularity of the inflated domain which will give us the compactness of 
$\overline\partial$, thus replacing the assumption on the compactness of the 
$\overline\partial$-Neumann operator.  As a corollary we obtain a weighted 
version of the Axler-Zheng theorem for strongly pseudoconvex domains, which 
itself is a new result.

\subsection{Preliminaries}
Let $\D$ be a $C^1$-smooth bounded pseudoconvex domain in $\C^n$ with a defining 
function $\rho$. We denote the boundary of $\D$ by  $b\D$. Let  
$L^2(\D,(-\rho)^r)$ denote the square integrable functions on $\D$ with respect 
to the measure $(-\rho)^rdV$ where $dV$ denotes the Lebesgue measure, 
$r\geq 0$, and 
\[A^2\left(\D,(-\rho)^r\right)=\left\{f\in L^2(\D,(-\rho)^r): f  \text{ is holomorphic}\right\}.\]
Since  $A^2\left(\D,(-\rho)^r\right)$ is a closed subspace of $L^2(\D,(-\rho)^r)$ a bounded 
orthogonal projection \[P_r:L^2(\D,(-\rho)^r)\to A^2(\D,(-\rho)^r),\] 
(called Bergman projection) exists. $P_r$ is an integral operator of the form 
\[P_r(f)(z)=\int_{\D}K^r(z,\zeta)f(\xi)(-\rho)^rdV\]
for $f\in L^2(\D,(-\rho)^r)$. The integral kernel $K^r(z,\xi)$ is called the 
Bergman kernel and the normalized Bergman kernel $k^r_z(\xi)$ is defined as 
$k_z^r(\xi)=\frac{K^r(\xi,z)}{\sqrt{K^r(z,z)}}$.
When $r=0$ we drop the superscript $r$; that is, $K=K_{\D}$ denotes the 
unweighted Bergman kernel and $k_z$ denotes the unweighted normalized Bergman 
kernel. For a bounded operator $T$ on $A^2\left(\D,(-\rho)^r\right)$, the 
Berezin transform $B_rT$ of $T$ is defined as 
\[B_rT(z)=\langle Tk_z^r,k_z^r\rangle_r\]
where  and $\langle .,.\rangle_r$ is the inner product on $A^2(\D,(-\rho)^r)$.

For $\phi\in L^{\infty}(\D)$, the weighted Toeplitz operator $T^r_{\phi}$ 
and the weighted Hankel operator $H^r_{\phi}$ are defined as follows 
\begin{align*}
T^r_{\phi}&=P^rM_{\phi}\\ 
H^r_{\phi}&=(I-P^r)M_{\phi}
\end{align*} 
where $M_{\phi}:A^2(\D,(-\rho)^r)\to L^2(\D,(-\rho)^r)$ denotes the 
multiplication by $\phi$.  

We use $\mathscr{T}(\Dc,(-\rho)^r)$ to denote the norm closed subalgebra of 
bounded linear operators on $A^2(\D,(-\rho)^r)$ generated by the set of 
Toeplitz operators $\{T^r_{\phi}:\phi\in C(\Dc)\}.$ For $\phi\in L^{\infty}$ we 
define  $B_r\phi=B_rT_{\phi}$.

In this paper we look at weighted Hankel and Toeplitz operators on  various domains 
and various weighted spaces. Whenever we need to clarify where these operators 
are defined, we will use appropriate subscripts and superscripts. In 
particular, when we need to emphasize the underlying domain we will write 
$P^{\D}, K_{\D}(z,\xi), H^{\D}_{\phi}$, and $T^{\D}_{\phi}$, where the Bergman 
spaces are unweighted. When we have weighted spaces and we need to indicate the 
domain and the weight we will write 
$P^{\D, r}, K_{\D}^r(z,\xi), H^{\D,r}_{\phi}$, and $T^{\D, r}_{\phi}$.

\subsection{Main Result}
We start with the following two definitions that capture the local structure 
of the main theorem. To motivate the following definition, if $f_j\to f$ weakly 
in $A^2(\D)$ then for any point $p\in b\D$ and $r>0$ one can show that $f_j\to f$ 
weakly in  $A^2(\D\cap B(p,r))$ where $B(p,r)$ is the open ball centered at $p$ with 
radius $r$. 

\begin{definition}
	Let $r\geq 0$ and $\D$ be a $C^2$-smooth 
	bounded pseudoconvex domain in $\C^n$ with a defining function $\rho$. 
	Furthermore, let $\{f_j\}\subset A^2(\D,(-\rho)^r)$ be a sequence and 
	$f\in A^2(\D,(-\rho)^r)$. We say  that $\{f_j\}$ converges to $f$ 
	\textit{weakly about strongly pseudoconvex points} if 
\begin{itemize}
\item[i.] $f_j\to f$ weakly in $A^2(\D,(-\rho)^r)$ as $j\to \infty$,
\item[ii.] in case $\Gamma_{\D}$, the set of the weakly  pseudoconvex 
points in $b\D$, is non-empty, there exists an open neighborhood $U$ of 
$\Gamma_{\D}$ such that $\|f_j-f\|_{L^2(U\cap \D,(-\rho)^r)}\to 0$ as 
$j\to \infty$. 
\end{itemize} 
\end{definition}

We note that on strongly pseudoconvex domains, sequences converging weakly  
about strongly pseudoconvex points and weakly convergent sequences coincide.

\begin{definition}
Let $r$, $\D$, and $\rho$ be as above. 
Furthermore, let  $T:A^2(\D,(-\rho)^r)\to A^2(\D,(-\rho)^r)$ be a bounded 
linear operator. We say that $T$ is \textit{compact about strongly pseudoconvex 
points} if $Tf_j\to Tf$ in $A^2(\D,(-\rho)^r)$ whenever $f_j\to f$ weakly about 
strongly pseudoconvex points.
\end{definition}

\begin{remark}
As shown in Proposition \ref{PropComp} below, it is interesting that any 
Hankel operator with a symbol continuous on the closure of the domain 
is compact about strongly pseudoconvex points. 
\end{remark}

With the help of these two definitions, we state our main result as follows.

\begin{theorem}\label{Thm1}
Let $r$ be a nonnegative real number, $\D$ be a $C^2$-smooth bounded pseudoconvex 
domain in $\C^n$ with a defining function $\rho$, and $T\in \mathscr{T}(\Dc,(-\rho)^r)$. 
Then $T$ is compact about strongly pseudoconvex points on 
$A^2(\D, (-\rho)^r)$ if and only if $\lim_{z\to p}B_rT(z)=0$ for any strongly 
pseudoconvex point $p\in b\D$.
\end{theorem}

If $\D$ is a strongly pseudoconvex domain then we have the following corollary.

\begin{corollary}\label{Cor1}
	Let $r$ be a nonnegative real number, $\D$ be a $C^2$-smooth bounded strongly 
	pseudoconvex domain in $\C^n$ with a defining function $\rho$, and 
	$T\in \mathscr{T}(\Dc,(-\rho)^r)$. Then $T$ is compact on $A^2(\D, (-\rho)^r)$ 
	if and only if $\lim_{z\to p}B_rT(z)=0$ for any $p\in b\D$.
\end{corollary}

\begin{remark}
	In the case of the unit ball $\mathbb{B}^n$ in $\mathbb{C}^n$ and $\rho(z)=|z|^2-1$, 
	we partially recover \cite[Theorem 1.1]{MitkovskiSuarezWick13}. Unlike the arguments 
	on the unit ball, the proof of Corollary \ref{Cor1} does not require any explicit form for 
	the weight or the weighted Bergman kernel. 
\end{remark}

\section{Proof of Theorem \ref{Thm1}}

In this section, before we prove Theorem \ref{Thm1}, we present some 
propositions and  lemmas that encapsulate the technical details of the proof.

\begin{proposition}\label{PropQuasiCompClosed}
	Let $\D$ be a $C^2$-smooth bounded pseudoconvex domain in $\C^n$ 
	and $\{T_j\}$ be a sequence of operators compact about strongly 
	pseudoconvex points  that converge to $T$ in the operator norm. Then $T$ is 
	compact about strongly pseudoconvex points. 
\end{proposition}
\begin{proof}
	Let $\{f_j\}$ be a sequence in $A^2(\D, (-\rho)^r)$ that converges to 0 weakly 
	about strongly pseudoconvex points.	 Since $f_j\to 0$ weakly there exists 
	$C>0$ such that 
	\[\sup\{\|f_j\|:j=1,2,3,\ldots\}\leq C.\] 
	Then for any $k$ we have 
	\[ \|Tf_j\|\leq \|(T-T_k)f_j\|+\|T_kf_j\|\leq C\|T-T_k\|+\|T_kf_j\|.\]
	Let $\ep>0$ be given. Since $T_j\to T$ in the operator norm, 
	we choose $k_{\ep}$ such that  $\|T-T_{k_{\ep}}\|\leq \ep$. Then 
	\[\limsup_{j\to \infty}\|Tf_j\|\leq C\ep+\limsup_{j\to\infty}\|T_{k_\ep}f_j\|\leq C\ep.\]
	Since $\ep>0$ was arbitrary we conclude that $Tf_j\to 0$. That is, $T$ is 
	compact about strongly pseudoconvex points.
\end{proof}

One of the key ideas in the proof is to use 
an inflated domain over $\D$ to understand the weighted Bergman spaces. For 
this purpose, unless stated otherwise, for the rest of the  paper, $\D$ will be 
a bounded pseudoconvex domain in $\C^n$ with $C^2$-smooth boundary, 
$\rho$ will be a defining function for $\D$, and 
\begin{align}\label{EqnDomain}
\D_r^p=\left\{(z,w)\in \C^n\times \C^p:z\in \D \text{ and }
\rho(z)+|w_1|^{2p/r} +\cdots+|w_p|^{2p/r}<0\right\} 
\end{align}
where $p$ is a positive integer and $r$ is a real number such that $0<r\leq p$. 
For a function $f\in A^2(\D, (-\rho)^r)$, we let $F(z,w)=f(z)$ be the 
trivial extension of $f$ to $\D^p_r$. It easily follows from an iterated 
integral argument that $F\in A^2(\D^p_r )$.

The following proposition is interesting in its own right as it gives a relationship 
between the Bergman kernels of the inflated domain and base. 

\begin{proposition}\label{PropKernel}
Using the notation above 
\[K_{\D}^r(z,\xi)=c_{p,r}K_{\D_r^p}(z,0;\xi,0)\] 
where $c_{p,r}=\int_{ |\widetilde{w}_1|^{2p/r}+\cdots+|\widetilde{w}_p|^{2p/r}< 1}
dV(\widetilde{w})$ and $K_{\D}^r(z,\xi)$  is the weighted Bergman kernel of $\D$ 
with weight $(-\rho)^r$.
\end{proposition}
\begin{proof}
We will follow a standard inflation argument (see for instance 
\cite{ForelliRudin74/75,Ligocka89}). Since $\D_r^p$ is a Hartogs domain with 
base $\D$, the Bergman kernel of $\D_r^p$ can be written as 
\[K_{\D_r^p}(z,w;\xi,\eta)=K_{\D_r^p}(z,0;\xi,0)+\sum_{|J|\geq 1}
K_J(z, \xi)w^J\overline{\eta}^J\]
where $J$ is a multiindex with nonnegative entries. 
Then for $f\in A^2(\D,(-\rho)^r)$ and $z\in \D$ we have 
($F$ below is the trivial extension of $f$)
\begin{align}\label{EqnExpansion}
f(z)=\int_{\D_r^p}K_{\D_r^p}(z,0;\xi,0)F(\xi,\eta)dV(\xi , \eta) 
+\sum_{|J|\geq1}\int_{\D_r^p}K_J(z,\xi)w^J\overline{\eta}^JF(\xi,\eta)
dV(\xi,\eta). 
\end{align}
However, the integrals under the sum on the right hand side above all vanish. 

Using the change of variables $\widetilde{w}_j= \frac{w_j}{(-\rho(z))^{r/2p}}$
one can compute that 
\begin{align}\label{EqnDilation}
	\int_{|w_1|^{2p/r}+\cdots+|w_p|^{2p/r}< -\rho(z)} dV(w)
	=(-\rho(z))^r\int_{|\widetilde{w}_1|^{2p/r}
+\cdots+|\widetilde{w}_p|^{2p/r}<1}dV(\widetilde{w}). 
\end{align}
	We denote 
\begin{align}\label{EqnDilation*}
c_{p,r}=\int_{|\widetilde{w}_1|^{2p/r}
+\cdots+|\widetilde { w } _p|^ { 2p/r } < 1}
dV(\widetilde{w}). 
\end{align}
Then using \eqref{EqnExpansion},\eqref{EqnDilation}, and 
\eqref{EqnDilation*}   we get 
\[ f(z)=\int_{\D_r^p}K_{\D_r^p}(z,0;\xi,0)F(\xi,\eta)dV(\xi,\eta)
= c_{p,r}\int_{\D}K_{\D_r^p}(z,0;\xi,0)f(\xi)(-\rho(\xi))^rdV(\xi).\]
Therefore,  $c_{p,r}K_{\D_r^p}(z,0;\xi,0)=K_{\D}^r(z,\xi)$.
\end{proof}

For a $C^2$-smooth function $\rho$ around a point 
$P\in \C^n, X=(x_1,\ldots,x_n)\in \C^n$, and 
$Y=(y_1,\ldots,y_n)\in \C^n$, we define the complex Hessian of $\rho$ at $P$ as 
\[H_{\rho}(P;X, Y)
=\sum_{j,k=1}^n \frac{\partial^2\rho (P)}{\partial z_j\partial \zb_k} x_j\overline{y_k}.\]
Furthermore, we use the notation $H_{\rho}(P;X)=H_{\rho}(P;X, X)$. 

\begin{lemma}\label{LemInflation}
Let $\D$ be a $C^2$-smooth bounded pseudoconvex domain in $\C^n$, 
$z_0\in b\D$ be a strongly pseudoconvex point, and $\D^p_r$ be defined 
as in \eqref{EqnDomain}. Then there exists $s>0$ such that $(z,w)\in b\D_r^p$ 
is strongly pseudoconvex for $|z-z_0|<s$ and $w_k\neq 0$ for all $1\leq k\leq p$.  
\end{lemma}
\begin{proof}
Let $\widetilde{\rho}(z,w)=\rho(z)+\lambda(w)$ where 
$\lambda(w)=|w_1|^{2p/r}+\cdots+ |w_p|^{2p/r}$ and $p\geq r$ an integer. 
Then $\widetilde{\rho}$ is a $C^2$-smooth function. 
Assume that $Q=(z,w)\in b\D_r^p$ is near $z_0$ and $X$ is a complex 
tangential vector to $b\D_r^p$ at $Q$. Then $X$ can be written as $X=X_n+X_p$ 
where $X_n$ and $X_p$ are the components of $X$ in the $z$ and $w$ variables, 
respectively. Then 
\[H_{\widetilde{\rho}}(Q;X)= H_{\rho}(z;X_n)+H_{\widetilde{\rho}}(Q;X_n, X_p)
+H_{\widetilde{\rho}}(Q;X_p,X_n)+H_{\lambda}(w;X_p).\]
However, $H_{\widetilde{\rho}}(Q;X_n, X_p)=H_{\widetilde{\rho}}(Q;X_p,X_n)=0$ as 
$z$ and $w$ are decoupled in $\widetilde{\rho}$. Then 
\[H_{\widetilde{\rho}}(Q;X)= H_{\rho}(z;X_n)+H_{\lambda}(w;X_p).\]
Let $\pi$ denote the projection from a neighborhood of $b\D$ in $\C^n$ onto 
$b\D$. Then  $X_n=X_t+X_{\nu}$ where $X_t$ is a tangential vector to $b\D$ at 
$\pi z$  and $X_{\nu}$ is a vector complex normal to $b\D$ at $\pi z$. Then 
\[H_{\rho}(z;X_n)=H_{\rho}(z;X_t)+H_{\rho}(z;X_t,X_{\nu})
+H_{\rho}(z;X_{\nu},X_t)+H_{\rho}(z;X_{\nu}).\]
We note that the complex Hessian $H_{\rho}$ changes continuously and 
$w\to 0$ as  $z\to z_0$ (here we assume that $(z,w)\in b\D_r^p$). 
Furthermore,  $X_{\nu}\to 0$ as $z\to z_0$ (as the complex normal 
to $b\D$ at $z_0$ is parallel to the complex normal to $b\D_r^p$ at $(z_0,0)$). 
Then, using the fact that $z_0$ is a strongly pseudoconvex point,  
we conclude that there exists $s>0$ so that 
 \[H_{\rho}(z;X_n)\geq \frac{H_{\rho}(\pi z;X_t)}{2} > 0\]
 for $|z-z_0|<s$ and $X_t\neq 0$. Also $H_{\lambda}(w;X_p)>0$ 
whenever  $X_p\neq 0$ and $w_k\neq 0$ for all $k$ as $\lambda$ is strictly 
plurisubharmonic whenever $w_k\neq 0$ for all $k$. Therefore, 
$H_{\widetilde{\rho}}(Q;X)>0$ for  $Q=(z,w)\in b\D_r^p$ 
such that $|z-z_0|<s$ and $w_k\neq 0$ for all $k$. 
\end{proof}

The following corollary follows from the previous lemma together with the 
fact that $\D^p_r$ has $C^2$-smooth boundary for $0<r\leq p$.
\begin{corollary}\label{CorInflation}
	Let $\D$ be a $C^2$-smooth bounded pseudoconvex domain in $\C^n$, 
	$z_0\in b\D$ be a strongly pseudoconvex point, and $\D^p_r$ be defined 
	as in \eqref{EqnDomain}. Then there exists $\ep>0$ such that 
	$B((z_0,0),\ep)\cap \D_r^p$ is pseudoconvex.
\end{corollary}

Next we will prove some statements about compactness of single Toeplitz 
and Hankel operators.

\begin{lemma}\label{LemSlice}
Let $\phi\in L^{\infty}(\D), \{f_j\}$ be a bounded sequence in  
$A^2(\D, (-\rho)^r)$ and  $F_j$ be the trivial extension of $f_j$ to $\D_r^p$ 
for each $j$ where $\D^p_r$ be defined as in \eqref{EqnDomain}. 
Assume that $\{H^{\D_r^p}_{\phi}F_j\}$ is convergent in  
$L^2(\D_r^p)$. Then $\{H^{\D,r}_{\phi}f_j\}$ is convergent in 
$L^2(\D, (-\rho)^r)$.  
\end{lemma}
\begin{proof}
We will abuse the notation and denote the trivial extension of $\phi$ to $\D_r^p$ 
by $\phi$. We assume that $\{H^{\D_r^p}_{\phi}F_j\}$ is convergent (and 
hence Cauchy). Let 
\[G_j(z,w)= (H^{\D_r^p}_{\phi}F_j)(z,w)\] 
and $g_j(z)=G_j(z,0)$. Then $G_j$ is holomorphic in $w$ because 
\[\frac{\partial G_j}{\partial \wb_k} 
=\frac{\partial }{\partial \wb_k}(I-P^{\D_r^p})(F_j\phi)
=\frac{\partial (F_j\phi) }{\partial \wb_k}=0\] 
for all $j$ and $1\leq k\leq p$. We note that 
$\frac{\partial (F_j\phi) }{\partial \wb_k}=0$ 
as $F_j\phi$ is independent of $w_k$. 
 Then $|G_j(z,w)-G_k(z,w)|^2$ is subharmonic in $w$ and using 
the mean value property for subharmonic functions together with 
\eqref{EqnDilation} and \eqref{EqnDilation*} one can show that 	
\begin{align*}\label{mean}
|g_j(z)-g_k(z)|^2\leq \frac{1}{c_{p,r} (-\rho(z))^r} \int_{|w_1|^{2p/r}+\cdots+|w_p|^{2p/r}
	< -\rho(z)}|G_j(z,w)-G_k(z,w)|^2dV(w)
\end{align*}
for $j=1,2,\ldots$ and $z\in \D$.  By integrating over $\D$ we get  	 
\[c_{p,r} \|g_j-g_k\|^2_{L^{2}_{(0,1)}(\D, (-\rho)^r)}\leq \|G_j-G_k\|^2_{L^{2}_{(0,1)}(\D_r^p)}\]
 for $j,k=1,2,\ldots.$ Then $\{g_j\}$ is a Cauchy sequence in $L^{2}_{(0,1)}(\D, (-\rho)^r)$  
 (and hence convergent) because $\|G_j-G_k\|_{L^{2}_{(0,1)}(\D_r^p)}\to 0$  
 as $j,k\to \infty$. 

Let $h_j(z)=P^{\D_r^p}(\phi F_j)(z,0)$. Then 
\[c_{r,p}\|h_j\|^2_{L^2(\D,(-\rho)^r)}\leq \|P^{\D_r^p}(\phi F_j)\|^2_{L^2(\D_r^p)} 
\leq \|\phi F_j\|^2_{L^2(\D_r^p)}=c_{r,p}\|\phi f_j\|^2_{L^2(\D,(-\rho)^r)}<\infty\]
for each $j$. Hence, $h_j\in A^{2}(\D, (-\rho)^r)$ and $(I-P^{\D,r})h_j =0$ for all $j$. 
We get equality between the last terms above because $F_j$ and $\phi$ are 
independent of $w$. Now 
\begin{align*}
(I-P^{\D,r})g_j&=(I-P^{\D,r})\left(\phi f_j-P^{\D_r^p}(\phi F_j)(.,0)\right)\\
&=(I-P^{\D,r})\left(\phi f_j\right)-(I-P^{\D,r})\left(h_j\right)\\
&=H^{\D,r}_{\phi}f_j.
\end{align*}
Therefore, the sequence $\{H^{\D,r}_{\phi}f_j\}$ is convergent in $L^2(\D, (-\rho)^r)$. 
\end{proof}

\begin{lemma}\label{LemComp}
	Let $r$ be a nonnegative real number and $\D$ be a $C^2$-smooth 
	bounded pseudoconvex domain in $\C^n$ with a defining function $\rho$. 
	Assume that $\phi\in C(\Dc)$ such that $\phi(z)=0$ if $z$ is a strongly 
	pseudoconvex point in $b\D$. Then $T^r_{\phi}$ is compact about 
	strongly pseudoconvex points  on $A^2(\D, (-\rho)^r)$.
\end{lemma}
\begin{proof}
	Let $\{f_j\}$ be a sequence in $A^2(\D, (-\rho)^r)$ that (without loss of generality) 
	converges to 0 weakly 
	about strongly pseudoconvex points.  Then $f_j\to 0$ weakly as $j\to \infty$ and 
	there is a neighborhood $U$ of weakly pseudoconvex points in $b\D$ such that 
	\[\|f_j\|_{L^2(U\cap \D,(-\rho)^r)}\to 0 \text{ as } j\to \infty.\] 
	Using the uniform boundedness principle and the fact that  $f_j\to 0$ 
	weakly we conclude that  the sequence  $\{f_j\}$ is bounded in 
	$A^2(\D, (-\rho)^r)$. Furthermore,  Cauchy estimates together with 
	Montel's Theorem (and the fact that $f_j\to 0$ weakly)  imply that $\{f_j\}$ 
	converges to zero uniformly on compact subsets of $\D$.  
	Using the fact that $\phi=0$ on strongly pseudoconvex points, one can show that 
	$\phi f_j\to 0$ in $A^2(\D, (-\rho)^r)$. Therefore, $T^r_{\phi}f_j\to 0$ in  
	$A^2(\D, (-\rho)^r)$. That is, $T^r_{\phi}$ is compact about strongly pseudoconvex 
	points on in $A^2(\D, (-\rho)^r)$. 
\end{proof}

Let $\D$ be a domain in $\C^n$. Then $z\in b\D$ is said to have a holomorphic 
(plurisubharmonic)  peak function if there exists a holomorphic (plurisubharmonic) 
$f$ that is continuous on $\Dc$  such that $f(z)=1$ and $|f(w)|<1$ ($f(w)<1$) 
for $w\in \Dc\setminus \{z\}$. 

Next we show that any Hankel operator with a symbol continuous on the closure 
of the domain is compact about strongly pseudoconvex points. 

\begin{proposition}\label{PropComp}
 Let $r$ be a nonnegative real number, $\D$ be a $C^2$-smooth bounded 
pseudoconvex domain in $\C^n$ with a defining function $\rho$, and $\phi\in 
C(\Dc)$. Then $H^r_{\phi}:A^2(\D, (-\rho)^r)\to L^2(\D, (-\rho)^r)$ is compact 
about strongly pseudoconvex points. 
\end{proposition}

\begin{proof}
We will prove more (see Corollary \ref{Cor2} below). First of all, for any 
$\phi\in C(\Dc)$ there exists $\{\phi_j\}\subset C^1(\Dc)$ such that 
$\phi_j\to \phi$ uniformly on $\Dc$ as $j\to \infty$. Furthermore, 
$\{H_{\phi_j}^{r}\}$ converges to $H_{\phi}^{r}$ in the operator norm 
and, by Proposition   \ref{PropQuasiCompClosed}, if $H_{\phi_j}^{r}$ 
is compact about strongly pseudoconvex point for every $j$ then so is $H_{\phi}^{r}$. 
Therefore, for the rest of the proof we will assume that $\phi\in C^1(\Dc)$.
Secondly, the proof for $r=0$ does not require the inflation argument 
	in the next paragraph and hence it is easier than the case $r>0$. Since both 
	proofs are similar, except for the inflation argument, in the rest of the proof, 
	we will assume that $r>0$.
 
Let $z_0\in b\D$ be a strongly pseudoconvex point. Then, by Corollary \ref{CorInflation}, 
the domain $B((z_0,0),\ep)\cap \D^p_r$ is pseudoconvex for small $\ep$. Let $\ep>0$ 
be such that $X_0=b\D\cap  \overline{B(z_0,\ep)}\subset \C^n$ consists of strongly 
pseudoconvex points. Let us define  
\[Y=b\D_r^p\cap  \overline{B(z_0,\ep)} 
\cap  \{(z,w)\in \C^n\times \C^p: w_k=0 \text{ for some } 1\leq k\leq p\}\]
\[X_j=b\D_r^p\cap  \overline{B(z_0,\ep)} 
	\cap  \{(z,w)\in \C^n\times \C^p:|w_k|\geq 1/j \text{ for all } 1\leq k\leq p\}\]
 for  $j=1,2,3,\ldots$. Then $X_0$ is B-regular as any point in $X_0$ has a 
holomorphic  (hence plurisubharmonic) peak function on $\D\subset \C^n$. 
The same function (by extending it trivially) is also a plurisubharmonic peak 
function on $\D_r^p\subset \C^{n+p}$. Hence, $X_0$ is B-regular as a compact 
set in $\C^{n+p}$. Furthermore, Lemma \ref{LemInflation} implies that we can 
shrink $\ep$, if necessary, so that $X_j$'s  are composed of strongly pseudoconvex 
points for $j\geq 1$. Hence, $X_j$ is B-regular for every $j=0,1,2,\ldots$.  

Next we will apply a similar idea to $Y$ in lower dimensions. 
Let us define $Y_1=\cup_{m=1}^pY_1^m$ where 
\[Y^m_1=b\D_r^p\cap  \overline{B(z_0,\ep)} 
\cap  \{(z,w)\in \C^n\times \C^p: w_k=0 \text{ for } k\neq m\}.\]
We can write $Y^m_1$ as the union of $X_0$ together with the compact sets  
\[b\D_r^p\cap  \overline{B(z_0,\ep)} 
\cap  \{(z,w)\in \C^n\times \C^p:|w_m|\geq 1/j, w_k=0 \text{ for } k\neq m\}\]
for $j=1,2,3,\ldots$. However, we can think of the sets above as subsets in 
$\C^n\times \C$ and (by Lemma \ref{LemInflation}) they are composed of 
strongly pseudoconvex points. Hence, they are B-regular. Then 
\cite[Proposition 1.9]{Sibony87} implies that each $Y^m_1$ is B-regular 
as it is a countable union of B-regular sets. Hence, applying Sibony's 
proposition again, we conclude that $Y_1$ is compact. Similarly, we can 
define $Y_2\subset Y$ as a countable union of compact sets where all 
but at most two $w_k$'s are equal to 0. Using the same reasoning above 
adopted for $Y_2$ we can conclude that $Y_2$ is B-regular. In a similar 
fashion, we can define $Y_l$ for $1\leq l\leq p-1$ and prove that all of 
them are B-regular. Hence $Y=(\cup_{l=1}^pY_l)\cup X_0$ is B-regular. Then 
\[b(\D_r^p\cap B(z_0,\ep))\subset Y\cup (\cup_{j=0}^{\infty}X_j)\cup bB(z_0,\ep)\]  
is B-regular (satisfies Property $(P)$ in Catlin's terminology) 
and, hence, the $\dbar$-Neumann operator on  $\D_r^p\cap B(z_0,\ep)$ is  
compact (see \cite[Theorem 4.8]{StraubeBook} and \cite{Catlin84}).  
Then $H_{\phi}^{\D_r^p\cap B(z_0,\ep)}$ is compact 
(see \cite[Proposition 4.1]{StraubeBook}) and  Lemma \ref{LemSlice} 
implies that $H_{\phi}^{\D\cap B(z_0,\ep),r}$ is compact.

Next we will use local compact solution operators to show that $H^r_{\phi}$ is 
compact about strongly pseudoconvex points. Let 
$\{f_j\}\subset A^2(\D,(-\rho)^r)$ be a sequence weakly convergent about 
strongly pseudoconvex points. Then there exists an open neighborhood $U$ of the 
set of weakly pseudoconvex points in $b\D$ such that
\begin{itemize}
 \item[i.] $\{f_j\}$ is weakly convergent,
 \item[ii.] $\|f_j-f_k\|_{L^2(U\cap \D,(-\rho)^r)}\to 0$ as $j,k\to \infty$. 
\end{itemize}
Let us choose $\{p_k:k=1,\ldots,m\}\subset b\D\setminus U$ and 
$\ep_k>0$ (for  $k=1,\ldots,m$) such that 
\begin{itemize}
 \item[i.] $b\D\setminus U\subset \cup_{k=1}^mB(p_k,\ep_k)$
 \item[ii.] $H_{\phi}^{k,r}=H_{\phi}^{B(p_k,\ep_k)\cap \D,r}$ is 
compact on $A^2(B(p_k,\ep_k)\cap \D,(-\rho)^r)$ for $k=1,\ldots,m$. 
\end{itemize}
Let us choose a strongly pseudoconvex domain $\D_{-1}\Subset \D$ 
and smooth cut-off functions $\chi_{-1}\in C^{\infty}_0(\D_{-1}), \chi_0\in 
C^{\infty}_0(U), $ and $\chi_k\in C^{\infty}_0(B(p_k,\ep))$ for $k=1,\ldots,m$ 
such that $\sum_{k=-1}^m\chi_k\equiv 1$ on $\Dc$.

Let $H_{\phi}^{-1,r}= H_{\phi}^{\D_{-1},r}, H_{\phi}^{0,r}=H_{\phi}^{U\cap \D,r},$ 
and $g_j=\sum_{k=-1}^m\chi_kH_{\phi}^{k,r}f_j$. We note that 
$H_{\phi}^{-1,r}$ is compact as $\D_{-1}\Subset \D$ is strongly pseudoconvex 
(and $\rho<0$ on the closure of $\D_{-1}$); 
$\{H_{\phi}^{0,r}f_j\}$ is convergent as $\{f_j\}$ is convergent in
$L^2(U\cap \D,(-\rho)^r)$; and by the previous part  of 
this proof, $H_{\phi}^{k,r}$ is compact for each $k=1,\ldots,m$. 
Therefore, $\{g_j\}$ is convergent in $L^2(\D,(-\rho)^r)$.  
Furthermore, 
\[\dbar g_j= f_j\dbar \phi
+\sum_{k=-1}^m \left(\dbar\chi_k\right)H_{\phi}^{k,r}f_j.\]
Then $\left\{\sum_{k=-1}^m \left(\dbar\chi_k\right)H_{\phi}^{k,r}f_j\right\}$ 
is a convergent sequence of $\dbar$-closed $(0,1)$-forms as both $\dbar g_j$ 
and $f_j\dbar \phi$ are $\dbar$-closed. Let 
$Z^r:L^2_{(0,1)}(\D,(-\rho)^r)\to L^2(\D,(-\rho)^r) $ be  a bounded linear  
solution operator to $\dbar$ (see \cite{Hormander65}). Let 
\[h_j=g_j-Z^r  \sum_{k=-1}^m \left(\dbar\chi_k\right)H_{\phi}^{k,r}f_j.\]
Then $\{h_j\}$ is convergent and $\dbar h_j=f_j\dbar\phi$. So by 
taking projection on the orthogonal complement of $A^2(\D,(-\rho)^r)$ 
we get  $(I-P^{r})h_j=H_{\phi}^rf_j$. Therefore, $\{H_{\phi}^rf_j\}$ is 
convergent.  
\end{proof}

 Using the proof of the proposition above we get the following corollary. 
\begin{corollary}\label{Cor2}
 Let $r$ be a nonnegative real number and $\D$ be a $C^2$-smooth 
 bounded pseudoconvex domain in $\C^n$ with a defining 
 function $\rho$.  Assume that $\D$ satisfies 
 property (P) of Catlin  (or B-regularity of Sibony). Then 
 \begin{itemize}
 	\item[i.] $\dbar$ has a compact 
 	solution operator on $K^2_{(0,1)}(\D, (-\rho)^r)$, the weighted $\dbar$-closed 
 	$(0,1)$-forms, 
 	  \item[ii.]  $H^r_{\phi}:A^2(\D, (-\rho)^r)\to L^2(\D, (-\rho)^r)$ is 
 	  compact for all $\phi\in C(\Dc)$. 
 \end{itemize}
\end{corollary}
\begin{proof}
	Since ii. follows from i. we will only prove i. By a theorem Diederich and 
	Forn\ae{}ss \cite{DiederichFornaess77} there exists a $C^2$-smooth 
	defining function $\rho_1$ and $0<\eta\leq 1$ such that $-(-\rho_1)^\eta$ 
	is a strictly plurisubharmonic exhaustion function for $\D$. 
	Since $\rho_1$ and $\rho$ are comparable on $\Dc$ it is enough to 
	prove that $\dbar$ has a compact solution operator on 
	$K^2_{(0,1)}(\D, (-\rho_1)^r)$. 
	
	Let $s=r/\eta\geq 0$ and $q$ be an integer such that $s\leq q$. We define 
\[\D^q_s=\left\{(z,w)\in\C^n\times \C^q: -(-\rho_1(z))^\eta+\lambda(w)<0\right\}\] 
	where $\lambda(w)=|w_1|^{2q/s}+\cdots +|w_q|^{2q/s}$. Then 
	$-(-\rho_1)^\eta+\lambda$ is a bounded $C^2$-smooth plurisubharmonic 
	 function and $\D^q_s$ is pseudoconvex.  Furthermore,  the first part of the 
	 proof of Proposition \ref{PropComp} shows that $\D^q_s$ satisfies property (P). 
	 	
	Let $\{f_j\}$ be a bounded sequence in  $K^2_{(0,1)}(\D, (-\rho_1)^r)$. 
	Then $\{F_j\}$ is a bounded sequence in  $K^2_{(0,1)}(\D^q_s)$. 
	As shown in the first part of this proof, $\D_s^q$ is a bounded 
	(not necessarily $C^2$-smooth) pseudoconvex domain with 
	property (P). Then $\{\dbar^*N^{\D_s^q}F_j\}$ has a convergent 
	subsequence in $L^2(\D_s^q)$ where $N^{\D_s^q}$ is the 
	$\dbar$-Neumann operator on $L^2_{(0,1)}(\D_s^q)$. 
	By the proof of Proposition \ref{PropKernel} and the fact that 
 $\dbar^*N^{\D_s^q}F_j$ is holomorphic in $w$, we conclude that  
	$\dbar^*N^{\D_s^q}F_j(.,0) \in L^2(\D, (-\rho_1)^r)$.  
	Furthermore, $\dbar \dbar^*N^{\D_s^q}F_j(.,0)=f_j$ for all $j$ 
	and $\{\dbar^*N^{\D_s^q}F_j(.,0)\}$ has a convergent 
	subsequence in $L^2(\D, (-\rho_1)^r)$.  Therefore,  $\dbar$ has a compact 
	solution operator $R \dbar^*N^{\D_s^q}E$ on $K^2_{(0,1)}(\D, (-\rho_1)^r)$ 
	where $E$ is the trivial extension operator and $R$ is the restriction 
	from $\D_s^q$ onto $\D$. 	
\end{proof}

The following Lemma is essentially contained in the proof of 
\cite[Proposition 1.3]{ArazyEnglis01}. We present it here for 
the convenience of the reader.

\begin{lemma}\label{LemArazyEnglis}
Let $r$ be a nonnegative real number, $\D$ be bounded 
domain in $\C^n$,  and $\phi\in C(\Dc)$.
Assume that $z_0\in b\D$ has a holomorphic peak function. Then
\[\lim_{z\to z_0} B_rT^r_{\phi}(z)=\phi(z_0).\] 
\end{lemma}
\begin{proof}
First, we prove that for any neighborhood $U$ of $z_0$ 
\begin{equation}\label{smallintegral}
\int_{\D\setminus U}\left|k_z^r(w)\right|^2\left(-\rho(w)\right)^rdV(w)\to 0 
\text { as } z\to z_0.
\end{equation}
Indeed, for given $U$ and $\ep>0$ first we choose a holomorphic peak function 
$g$ such that $|g(w)|\leq \ep$ for all $w\in \D\setminus U$. This can be simply 
done by taking a high enough power of the holomorphic peak function $g$. 
Then we choose $\delta>0$ such that if $|z-z_0|<\delta$ and $z\in \D$ then 
$|g(z)|>1-\ep$. In this case,
\begin{align*}
\int_U \left|k_z^r(w)\right|^2\left(-\rho(w)\right)^rdV(w)
\geq& \int_U |g(w)|\left|k_z^r(w)\right|^2\left(-\rho(w)\right)^rdV(w)\\
\geq&\left| \int_{\D} g(w)\left|k_z^r(w)\right|^2\left(-\rho(w)\right)^rdV(w) \right|\\
&-\left|\int_{\D\setminus U} g(w)\left|k_z^r(w)\right|^2\left(-\rho(w)\right)^rdV(w)\right|\\
\geq &|g(z)|-\int_{\D\setminus U} |g(w)|\left|k_z^r(w)\right|^2\left(-\rho(w)\right)^rdV(w)\\
\geq& 1-\ep-\ep\int_{\D\setminus U}
 \left|k_z^r(w)\right|^2\left(-\rho(w)\right)^rdV(w)\\
\geq& 1-2\ep
\end{align*}
whenever $|z-z_0|<\delta$. This implies that for a given neighborhood $U$ and 
$\ep>0$, there exists $\delta>0$ such that if $|z-z_0|<\delta$ then 
\[\int_{\D\setminus U}  \left|k_z^r(w)\right|^2\left(-\rho(w)\right)^rdV(w)\leq \ep.\] 
This gives \eqref{smallintegral}.

Now for $\ep>0$, we choose a neighborhood $U$ of $z$ such that 
$|\phi(w)-\phi(z_0)|\leq \ep$ for all $w\in U$. Then for this neighborhood 
$U$ and the same $\ep$ we choose $\delta>0$ such that if $|z-z_0|<\delta$ 
then $\int_{\D\setminus U} 	\left|k_z^r(w)\right|^2\left(-\rho(w)\right)^rdV(w)
\leq \frac{\ep}{1+2\|\phi\|_{L^\infty}}.$ In this case,
	\begin{align*}
		\left|B_rT^r_{\phi}(z)-\phi(z_0)\right|\leq &\int_{\D} 
		|\phi(w)-\phi(z_0)|\left|k_z^r(w)\right|^2\left(-\rho(w)\right)^rdV(w)\\
		=&\int_{U}|\phi(w)-\phi(z_0)|
		  \left|k_z^r(w)\right|^2\left(-\rho(w)\right)^rdV(w)\\
		&+\int_{\D\setminus U} |\phi(w)-\phi(z_0)|\left|k_z^r(w)\right|^2\left(-\rho(w)\right)^rdV(w)\\
		\leq& \ep \int_{U} 
		\left|k_z^r(w)\right|^2\left(-\rho(w)\right)^rdV(w) \\ 
		&+2\|\phi\|_{L^\infty}\int_{\D\setminus U} 
		\left|k_z^r(w)\right|^2\left(-\rho(w)\right)^rdV(w)\\
		\leq& \ep+ \ep=2\ep.
	\end{align*}
	This indeed concludes $\lim_{z\to z_0} B_rT^r_{\phi}(z)=\phi(z_0)$.
\end{proof}

We note that on any bounded domain, we have 
(see \cite[Lemma 1]{CuckovicSahutoglu14}) 
\[T_{\phi_2}^rT_{\phi_1}^r 
=T_{\phi_2\phi_1}^r-H_{\overline{\phi}_2}^{r*}H_{\phi_1}^r.\]
Using the fact above inductively one can prove the following lemma.

\begin{lemma}\label{LemConvert}
Let $r$ be a nonnegative real number and $\D$ be a $C^1$-smooth bounded 
domain in $\C^n$ with a defining function $\rho$. 
Supposed $\phi_1,\ldots,\phi_m\in L^{\infty}(\D)$. Then
	\begin{align*}
	T_{\phi_m}^rT_{\phi_{m-1}}^r\cdots T_{\phi_2}^rT_{\phi_1}^r 
	= &T_{\phi_m\phi_{m-1}\cdots\phi_2\phi_1}^r - 
	T_{\phi_m}^rT_{\phi_{m-1}}^r \cdots T_{\phi_3}^r 
	H_{\overline{\phi}_2}^{r*}H_{\phi_1}^r  \\
	&- T_{\phi_m}^rT_{\phi_{m-1}}^r \cdots T_{\phi_4}^r
	H_{\overline{\phi}_3}^{r*}H_{\phi_2\phi_1}^r  -\cdots 
	-H_{\overline{\phi}_m}^{r*}H_{\phi_{m-1}\cdots\phi_2\phi_1}^r \\
	=&T_{\phi_m\phi_{m-1}\cdots\phi_2\phi_1}^r+S^r
	\end{align*}
where $S^r$ is a finite sum of finite products of operators and each product 
starts  with a Hankel operator.
\end{lemma}
Therefore, if the symbols  $\phi_1,\ldots,\phi_m$ are continuous on $\Dc$ we can write 
\begin{align}\label{EqnReduction}
T_{\phi}^r\cdots T_{\phi_m}^r=T_{\phi_1\cdots 
\phi_m}^r+S^r 
\end{align}
where $S^r$ is a finite sum of finite products of operators such that  
each product starts  with a Hankel operator with symbol continuous on $\Dc$.

We state the lemma below for general weights $\mu(z)$ (not only the ones of the 
form $(-\rho)^k$) that are nonnegative (can vanish on the boundary) and continuous 
on $\Omega$. The weights of this form are called admissible weights 
(see \cite{Pas90}) and the corresponding weighted Bergman projections 
and kernels are well defined. We say two weights $\mu_1$ and $\mu_2$ are 
comparable if there exists $c>0$ such that $c^{-1}\mu_1<\mu_2<c\mu_1$ on 
$\Omega$.

\begin{lemma} \label{LemEquiv}
	Let $\D$ be a domain in $\C^n$ and $\mu_1$ and $\mu_2$ be comparable 
      admissible weights. Let $k^{\mu_j}_z$ be the normalized Bergman kernel 
      corresponding to $\mu_j$ for $j=1,2$ and  $z_0\in b\D$. Then 
	$k^{\mu_1}_z\to 0$ weakly as $z\to z_0$ if and only if  
	$k^{\mu_2}_z\to 0$ weakly as $z\to z_0$.
\end{lemma}
\begin{proof}
	It is enough to show one direction. So we will showed that if $k^{\mu_1}_z\to 0$ 
	weakly as $z\to z_0$ then  $k^{\mu_2}_z\to 0$ weakly as $z\to z_0$. 
	Since $\mu_1$ and $\mu_2$ are equivalent measures we have 
	$A^2(\D, d\mu_1)=A^2(\D, d\mu_2)$ and there exists $C>1$ such that 
	\[\frac{\|f\|_{\mu_1}}{C}\leq \|f\|_{\mu_2}\leq C\|f\|_{\mu_1}\] 
	for all $f\in A^2(\D, d\mu_1)$. We remind the reader that for $z\in \D$ we have 
	\[K_{\mu_j}(z,z)=\sup \{|f(z)|^2:\|f\|_{\mu_j}\leq 1\}\]
	where $K_{\mu_j}$ is the Bergman kernel corresponding to $\mu_j$. Then 
	$K_{\mu_1}$ and $K_{\mu_2}$ are equivalent on the diagonal in the 
	 sense that there exists $D=C^2>1$ such that 
	\[\frac{K_{\mu_1}(z,z)}{D} \leq K_{\mu_2}(z,z)\leq DK_{\mu_1}(z,z). \]
	Now we assume that $k^{\mu_1}_z\to 0$ weakly as $z\to z_0$. 
	Let us fix $f\in A^2(\D, d\mu_1)$. Then  we have 
	\[\frac{f(z)}{\sqrt{K_{\mu_1}(z,z)}} 
	=\langle f,k^{\mu_1}_z\rangle_{\mu_1} \to 0 \text{ as } z\to z_0.\]
	Then 
	\[\langle f,k^{\mu_2}_z\rangle_{\mu_2} 
	=\frac{f(z)}{\sqrt{K_{\mu_2}(z,z)}}\to 0 \text{ as } z\to z_0.\]
	Therefore, we showed that if $k^{\mu_1}_z\to 0$ weakly as $z\to z_0$ 
	then  $k^{\mu_2}_z\to 0$ weakly as $z\to z_0$. 
\end{proof}

Let $\D$ be a pseudoconvex domain in $\C^n$ and $z_0\in b\D$. Then we call 
$z_0$ a \textit{bumping point} if for any $\delta>0$ there exists a 
pseudoconvex domain $\D_1$ such that $\{z_0\}\cup \D\subset \D_1\subset \D\cup 
B(z_0,\delta)$. 

\begin{lemma}\label{LemWeakConv}
	Let $r$ be a nonnegative real number, $\D$ be a bounded pseudoconvex  
	domain in $\C^n$ with Lipschitz boundary, and $z_0\in b\D$ be a bumping 
	point.  Then $k^r_z\to 0$ weakly as $z\to z_0$. 
\end{lemma}
\begin{proof}
	By Lemma \ref{LemEquiv}, without loss of generality, we assume that 
	$\rho$ denotes  the negative distance to the boundary of $\D$. 
	 
	Let us fix $f\in A^2(\D,(-\rho)^r)$ and choose $r_1,r_2>0$  so that 
	$0<r_1<r_2$ and the outward unit vector $\nu$ is transversal to 
	$B(z_0,2r_2)\cap b\D$. Since $z_0$ is a bumping point we choose a bounded 
	pseudoconvex domain $\D_1$ such that 
	\[\{z_0\}\cup \D\subset  \D_1\subset \D\cup B(z_0,r_1).\] 
	So even though $\D_1$ contains a small neighborhood of $z_0$, we have 
	$\D\setminus B(z_0,r_1)= \D_1\setminus B(z_0,r_1)$. 
	
	Let us choose $\chi\in C^{\infty}_0(B(z_0,r_2))$ such that $\chi\equiv 1$ on 
	a neighborhood of  $\overline{B(z_0,r_1)}$. For $\ep>0$ small we define  
	$f_{\ep}(z)=f(z-\ep\nu)$  and $g_{\ep}=(1-\chi)f+\chi f_{\ep}$. Then 
	\begin{itemize}
		\item [i.]  $f_{\ep}\in A^2(\D\cap B(z_0,r_2),(-\rho)^r)$ and $f_{\ep}\to f$ 
		in $L^2(\D\cap B(z_0,r_2),(-\rho)^r)$, 
		\item[ii.] 	$g_{\ep}|_{\D\cap B(z_0,r_2)}$ is $C^{\infty}$-smooth 
		and $g_{\ep}\to f$ in $L^2(\D,(-\rho)^r)$ as $\ep\to 0$. 
	\end{itemize}
	
Let	$\rho_1$ and $Supp (\dbar\chi)$ denote  the negative distance to the 
boundary of $\D_1$ and the support of $\dbar \chi$, respectively.  Then 
$Supp(\dbar\chi) \cap \D=Supp(\dbar\chi) \cap \D_1$ and $-\rho$ and $-\rho_1$ 
are equivalent  on the support of $\dbar \chi$. 
	Furthermore, $\dbar g_{\ep}$ is a $\dbar$-closed $(0,1)$-form on $\D_1$ 
	($\dbar g_{\ep}$ is well defined on $\D_1$ as $\dbar\chi =0$ on $B(z_0,r_1)$) 
	for all small $\ep>0$ and there exists $C>0$ such that 
	\[\|\dbar g_{\ep}\|_{L^2(\D_1,(-\rho_1)^r)}
	\leq C \|f-f_{\ep}\|_{L^2(\D\cap B(z_0,r_2),(-\rho)^r)}
	\|\dbar \chi\|_{L^{\infty}(B(z_0,r_2))}\to 0 \text{ as }\ep\to 0.\] 
	
Next we will use H\"{o}rmander's theorem \cite{Hormander65} with 
the plurisubharmonic exponential weight $-r\log(-\rho_1)$. We note 
that  $-\log(-\rho_1)$ is plurisubharmonic because $\D_1$ is pseudoconvex. 
Then using H\"{o}rmander's theorem we get a constant $c_{\D_1}>0$ 
(depending on $\D_1$) and $h_{\ep}\in L^2(\D_1)$ 
	such that $\dbar h_{\ep}=\dbar g_{\ep}$ and 
$\|h_{\ep}\|_{L^2(\D_1,(-\rho_1)^r)}
\leq c_{\D_1}\|\dbar g_{\ep}\|_{L^2(\D_1,(-\rho_1)^r)}$. 
 Furthermore, since $\dbar$ is elliptic on the interior and $\dbar g_{\ep}$ is 
  $C^{\infty}$-smooth on $\D_1$,  we have $h_{\ep}\in C^{\infty}(\D_1)$. 
	
	We define $\widetilde{f}_n=g_{1/n}-h_{1/n}$. Then we have 
	\begin{itemize}
		\item[i.] $\widetilde{f}_n\in A^2(\D,(-\rho)^r)$ and $\widetilde{f}_n\to f$ in 
		$A^2(\D,(-\rho)^r)$,
		\item[ii.] $\widetilde{f}_n|_{\D\cap B(z_0,r_1)}\in C^{\infty}(\overline{\D\cap 
		B(z_0,r_1)})$.
	\end{itemize}
	So $\{\widetilde{f}_n\}$ is a sequence converging to $f$ and each member of 
	the sequence is smooth up to the boundary of $\D$ on a neighborhood of $z_0$.
	
	Finally, we will show weak convergence of $k^r_z$ to 0 as $z\to z_0$. 
\begin{align*}
\left|\langle f,k^r_z\rangle_{A^2(\D,(-\rho)^r)}\right| 
	\leq& \left|\langle f-\widetilde{f}_n,k^r_z\rangle_{A^2(\D,(-\rho)^r)}\right| 
	+\left|\langle \widetilde{f}_n,k^r_z\rangle_{A^2(\D,(-\rho)^r)}\right| \\
	\leq& \|f-\widetilde{f}_n\|_{L^2(\D,(-\rho)^r)}+\frac{|\widetilde{f}_n(z)|}{\sqrt{K_r(z,z)}}.
\end{align*}
	The first term on the right hand side can be made arbitrarily small for large 
	enough $n$, because $\|f-\widetilde{f}_n\|_{L^2(\D,(-\rho)^r)}\to 0$ as $n\to \infty$. 
	So for $\delta>0$ given we choose $n_{\delta}$ so that 
	$\|f-\widetilde{f}_{n_{\delta}}\|_{L^2(\D,(-\rho)^r)}\leq \delta$. Then since 
	$\widetilde{f}_{n_{\delta}}$ is $C^{\infty}$-smooth on $\overline{\D\cap B(z_0,r_1)}$ 
	(and $K_r(z,z)\to \infty$ as $z\to z_0$) we conclude that  
	$|\widetilde{f}_{n_{\delta}}(z)|/\sqrt{K_r(z,z)}\to 0$ as  $z\to z_0$.  Hence, 
	$\limsup_{z\to z_0}|\langle f,k^r_z\rangle|\leq \delta$ for arbitrary $\delta>0$. 
	Therefore, $k^r_z\to 0$ weakly as $z\to z_0$. 
\end{proof}

Now we are ready to prove Theorem \ref{Thm1}.
\begin{proof}[Proof of Theorem \ref{Thm1}]
In case  $r=0$, the proof of the theorem simplifies greatly as inflation and 
the related techniques are unnecessary. So we will prove the more difficult 
case, $r>0$.  

First we assume that $T$ is compact about strongly pseudoconvex points. 
Let $\D_r^p$ be defined as in \eqref{EqnDomain} and $z_0\in b\D$ be a 
strongly pseudoconvex point. Since small $C^2$-perturbations of strongly 
pseudoconvex points stay pseudoconvex, $z_0$ is a bumping point 
for $\D$. Then Lemma \ref{LemWeakConv}  implies that $k^r_z\to 0$  
weakly as $z\to z_0$. Furthermore, there exists an  open neighborhood $U$  
of $z_0$ such that weakly pseudoconvex points are contained in 
$b\D\setminus \overline{U}$; and, as in the proof of 
\eqref{smallintegral},  one can show that 
\[\|k^r_z\|_{L^2(\D\setminus \overline{U},(-\rho)^r)}\to 0 \text{ as } z\to z_0.\]
Therefore, $\{k^r_z\}$ converges to $0$ weakly about 
strongly pseudoconvex points as $z\to z_0$. 
Moreover, since $T$ is compact about strongly pseudoconvex points 
(such operators map sequences of holomorphic functions weakly 
convergent about strongly pseudoconvex points  to convergent 
sequences) we conclude that 
\[B_rT(z)=\left\langle Tk^r_z, k^r_z\right\rangle_{A^2(\D,(-\rho)^r)}\to 0\] 
as $z\to z_0$. 
	
Next we prove the other direction. As a first step we assume that $T$ is a 
finite sum of finite products of Toeplitz operators on $A^2(\D,(-\rho)^r)$
with symbols continuous on $\Dc$. Furthermore, we assume that  
\[\lim_{z\to z_0}B_rT(z)= 0\] 
for any strongly pseudoconvex point $z_0\in b\D$.  

Lemma \ref{LemConvert} implies that  
\begin{align}\label{Eqn1}
T=T_{\phi}^r+S^r
\end{align}
where $\phi\in C(\Dc)$ and $S^r$  is a sum of operators that start 
with a Hankel operator with symbol continuous on $\Dc$. 

Lemma \ref{LemArazyEnglis} implies that  
\begin{align}\label{Eqn2}
\lim_{z\to z_0} B_rT^r_{\phi}(z)= \phi(z_0)
\end{align} 
as strongly pseudoconvex points have holomorphic peak functions 
(see, \cite[Theorem 1.13 in Ch VI]{RangeBook}). 

By Proposition \ref{PropComp}, the operator $H^r_{\psi}$ is compact 
about strongly pseudoconvex points for any $\psi\in C(\Dc)$.  
Then $H^r_{\psi}k^r_z\to 0$ as $z\to z_0$ for any $\psi\in C(\Dc)$ because,  
as proven in the first part of this proof, $k^r_z\to 0$ weakly about strongly 
pseudoconvex points as $z\to z_0$. Hence, $B_rS^r(z)\to 0$ as $z\to z_0$. 
Combining this with \eqref{Eqn1} and \eqref{Eqn2} we can conclude that 
\[\phi(z_0)=\lim_{z\to z_0}B_rT(z)= 0.\] 
Since $z_0$ was an arbitrary strongly pseudoconvex point, we have  $\phi=0$ 
on all the strongly pseudoconvex boundary points. Then Lemma \ref{LemComp} 
and the fact that $S^r$ is compact about strongly pseudoconvex points imply 
that $T$ is compact about strongly pseudoconvex points.

Finally, we assume that $T\in \mathscr{T}(\Dc,(-\rho)^r)$. Then, using 
Lemma \ref{LemConvert}, for every $\ep>0$ there exists 
$\phi_{\ep}\in C(\Dc)$ and an operator $S^r_{\ep}$,  compact 
about strongly pseudoconvex points, such that  
\[\|T+T_{\phi_{\ep}}^r+S^r_{\ep}\|\leq \ep.\] 
Then for $z\in \D$ we have 
\begin{align*}
\left|B_rT(z)+B_rT^r_{\phi_\ep}(z)+B_rS^r_{\ep}(z)\right|
&=\left|\langle Tk_z^r+T^r_{\phi_\ep}k_z^r+S^r_{\ep}k_z^r,k_z^r 
\rangle_r\right|\\
&\leq \|T+T_{\phi_{\ep}}^r+S^r_{\ep}\|\\
&\leq \ep.
\end{align*}
Since $B_rS^r_{\ep}(z)\to 0$ and  $B_rT^r_{\phi_\ep}(z)\to \phi_\ep(z_0)$
(and we assume that $B_rT(z)\to 0$  as $z\to z_0$)  as $z\to z_0$ we have 
$\left|\phi_\ep(z_0)\right| \leq \ep$. That is, $|\phi_{\ep}|\leq \ep$ on  
strongly pseudoconvex points of $\D$. We choose $\psi_{\ep}\in C(\Dc)$ such 
that $\psi_{\ep}=0$ on strongly pseudoconvex boundary points of $\D$ and 
\[\sup\{|\psi_{\ep}(z)-\phi_{\ep}(z)|:z\in \Dc\}\leq 2\ep.\]
Then Lemma \ref{LemComp} implies that $T^r_{\psi_{\ep}}$ is compact 
about strongly  pseudoconvex points and   
\[\| T^r_{\phi_\ep}-T^r_{\psi_{\ep}}\|\leq 2\ep.\]
Hence 
 \[\|T+T^r_{\psi_{\ep}}+S_{\ep}^r\|\leq \| T+T_{\phi_{\ep}}^r+S^r_{\ep}\|
 +\| T^r_{\psi_{\ep}}-T^r_{\phi_\ep}\|\leq 3\ep.\]
 Therefore, $T$ is in the norm closure of the compact about strongly 
pseudoconvex points operators. Finally, Proposition \ref{PropQuasiCompClosed} 
implies that $T$ is compact about strongly pseudoconvex points. 
\end{proof}

\section{Acknowledgment}
Part of this work was done while the second author was visiting Sabanc\i{}  
University. He thanks this institution for its hospitality and good working 
environment. We would like to thank the anonymous referee for pointing 
out and helping to fix some inaccuracies.
\singlespace


\begin{thebibliography}{MSW13}
	
	\bibitem[AE01]{ArazyEnglis01}
	J.~Arazy and M.~Engli{\v{s}}, \emph{Iterates and the boundary behavior of the
		{B}erezin transform}, Ann. Inst. Fourier (Grenoble) \textbf{51} (2001),
	no.~4, 1101--1133.
	
	\bibitem[AZ98]{AxlerZheng98}
	Sheldon Axler and Dechao Zheng, \emph{Compact operators via the {B}erezin
		transform}, Indiana Univ. Math. J. \textbf{47} (1998), no.~2, 387--400.
	
	\bibitem[Cat84]{Catlin84}
	David W. Catlin, \emph{Global regularity of the $\overline\partial$-Neumann 
	problem}, Complex analysis of several variables (Madison, Wis.,1982), 
	Proc. Sympos. Pure Math., vol.~41, Amer. Math. Soc., Providence, RI,
	1984, pp.~39--49.
	
	\bibitem[{\u{C}}{\c{S}}13]{CuckovicSahutoglu13}
	\u{Z}eljko \u{C}u\u{c}kovi{\'c} and S{\"o}nmez \c{S}ahuto\u{g}lu,
	\emph{Axler-{Z}heng type theorem on a class of domains in $\mathbb{C}^n$},
	Integral Equations Operator Theory \textbf{77} (2013), no.~3, 397--405.
	
	\bibitem[{\v{C}}{\c{S}}14]{CuckovicSahutoglu14}
	\u{Z}eljko \u{C}u\u{c}kovi{\'c} and S{\"o}nmez \c{S}ahuto\u{g}lu,
	\emph{Compactness of products of {H}ankel operators on convex {R}einhardt
		domains in {$\mathbb{C}^2$}}, New York J. Math. \textbf{20} (2014), 627--643.
	
	\bibitem[DF77]{DiederichFornaess77}
	Klas Diederich and John~Erik Fornaess, \emph{Pseudoconvex domains: bounded
		strictly plurisubharmonic exhaustion functions}, Invent. Math. \textbf{39}
	(1977), no.~2, 129--141.
	
	\bibitem[Eng99]{Englis99}
	Miroslav Engli{\v{s}}, \emph{Compact Toeplitz operators via the Berezin
		transform on bounded symmetric domains}, Integral Equations Operator Theory
	\textbf{33} (1999), no.~4, 426--455.
	
	\bibitem[FR75]{ForelliRudin74/75}
	Frank Forelli and Walter Rudin, \emph{Projections on spaces of holomorphic
		functions in balls}, Indiana Univ. Math. J. \textbf{24} (1974/75), 593--602.
	
	\bibitem[H{\"o}r65]{Hormander65}
	Lars H{\"o}rmander, \emph{$L^2$ estimates and existence theorems for the
		$\overline\partial $ operator}, Acta Math. \textbf{113} (1965), 89--152.
	
	\bibitem[Kre14]{KreutzerThesis}
	Elena Kreutzer, \emph{Toeplitz extensions and Berezin transforms}, Master
	thesis, Universit\"{a}t des Saarlandes, 2014.
	
	\bibitem[Lig89]{Ligocka89}
	Ewa Ligocka, \emph{On the Forelli-Rudin construction and weighted Bergman
		projections}, Studia Math. \textbf{94} (1989), no.~3, 257--272.
	
	\bibitem[MSW13]{MitkovskiSuarezWick13}
	Mishko Mitkovski, Daniel Su{\'a}rez, and Brett~D. Wick, \emph{The essential
		norm of operators on $A^p_\alpha(\mathbb{B}_n)$}, Integral Equations
	Operator Theory \textbf{75} (2013), no.~2, 197--233.

	\bibitem[PW90]{Pas90}
	Zbigniew Pasternak-Winiarski, \emph{On the dependence of the reproducing kernel
  	on the weight of integration}, J. Funct. Anal. \textbf{94} (1990), no.~1,
  	110--134. \MR{MR1077547 (91j:46035)}
  
	\bibitem[Ran86]{RangeBook}
	R. Michael Range, \emph{Holomorphic functions and integral representations in
		several complex variables}, Graduate Texts in Mathematics, vol. 108,
	Springer-Verlag, New York, 1986.
	
	\bibitem[Sib87]{Sibony87}
	Nessim Sibony, \emph{Une classe de domaines pseudoconvexes}, Duke Math. J.
	\textbf{55} (1987), no.~2, 299--319.
	
	\bibitem[Str10]{StraubeBook}
	Emil J. Straube, \emph{Lectures on the $\mathcal{L}^2$-Sobolev theory of
		the $\overline\partial$-Neumann problem}, ESI Lectures in Mathematics and
	Physics, vol.~7, European Mathematical Society (EMS), Z\"urich, 2010.
	
	\bibitem[Su{\'a}07]{Suarez07}
	Daniel Su{\'a}rez, \emph{The essential norm of operators in the {T}oeplitz
		algebra on $A^p(\mathbb{B}_n)$}, Indiana Univ. Math. J. \textbf{56} (2007),
	no.~5, 2185--2232.
	
\end{thebibliography}
\end{document}